\documentclass{amsart}[12pt]

\newtheorem{theorem}{Theorem}[section]
\newtheorem{lemma}[theorem]{Lemma}
\newtheorem{proposition}[theorem]{Proposition}

\newtheorem{problem}[theorem]{Problem}

\theoremstyle{definition}
\newtheorem{definition}[theorem]{Definition}

\newtheorem{remark}[theorem]{Remark}

\usepackage{amscd,amssymb}

\begin{document}

\title[Varieties of bicommutative algebras]
{Noetherianity and Specht problem\\
for varieties of bicommutative algebras}
\author{Vesselin Drensky, Bekzat K. Zhakhayev}
\date{}
\address{Institute of Mathematics and Informatics,
Bulgarian Academy of Sciences,
Acad. G. Bonchev Str., Block 8,
1113 Sofia, Bulgaria}
\email{drensky@math.bas.bg}
\address{Faculty of Engineering and Natural Sciences,
S\"uleyman Demirel University, 040900 Kaskelen, Almaty,
Kazakhstan}
\email{bekzat22@hotmail.com}

\thanks
{The research of the first named author was partially supported by Grant I02/18
``Computational and Combinatorial Methods
in Algebra and Applications''
of the Bulgarian National Science Fund.}
\thanks
{The research of the second named author was supported by Grant No. 0828/GF4 of the Ministry of Education and Science
of the Republic of Kazakhstan.}

\subjclass[2010]
{17A30; 17A50.}
\keywords{Free bicommutative algebras, varieties of bicommutative algebras, weak noetherianity, Specht problem.}
\maketitle

{\bf Abstract.}
Nonassociative algebras satisfying the polynomial identities
\[
x_1(x_2x_3)=x_2(x_1x_3)\text{ and }(x_1x_2)x_3=(x_1x_3)x_2
\]
are called bicommutative.
We prove the following results: (i) Finitely generated bicommutative algebras are weakly noetherian,
i.e., satisfy the ascending chain condition for two-sided ideals.
(ii) We give the positive solution to the Specht problem (or the finite basis problem)
for varieties of  bicommutative algebras over an arbitrary field of any characteristic.

\section{Introduction}

Let $K$ be a field. A $K$-algebra $R$ is called {\it right-commutative} if it satisfies the polynomial identity
\[
(x_1x_2)x_3=(x_1x_3)x_2,
\]
i.e., $(r_1r_2)r_3=(r_1r_3)r_2$ for all $r_1,r_2,r_3\in R$. Similarly one defines {\it left-commutative} algebras as algebras which satisfy the identity
\[
x_1(x_2x_3)=x_2(x_1x_3).
\]
Algebras which are both right- and left-commutative are called {\it bicommutative}.
We denote by $\mathfrak B$ the variety of all bicommutative algebras, i.e., the class of all algebras satisfying the identities of
right- and left-commutativity.
The first example of a  one-sided commutative algebra is the right-symmetric Witt algebra $W_1^{\text{rsym}}$ in one variable
which appeared already in the paper by Cayley \cite{Ca} in 1857,
\[
W_1^{\text{rsym}}=\left\{f\frac{d}{dx}\mid f\in K[x]\right\}
\]
equipped with the multiplication
\[
\left(f_1\frac{d}{dx}\right)\ast\left(f_2\frac{d}{dx}\right)
=\left(f_2\frac{df_1}{dx}\right)\frac{d}{dx}
\]
which is left-commutative. (Right-symmetric algebras satisfy the polynomial identity $(x_1,x_2,x_3)=(x_1,x_3,x_2)$, where
$(x_1,x_2,x_3)=(x_1x_2)x_3-x_1(x_2x_3)$ is the associator.)
Cayley also considered the realization of the right-symmetric Witt algebras $W_d^{\text{rsym}}$ in terms of rooted trees.
The algebra $W_1^{\text{rsym}}$ is an example of a Novikov algebra (which is right-symmetric and left-commutative).
(The algebras $W_d^{\text{rsym}}$ are not Novikov for $d>1$ because are right-symmetric but are not left-commutative.)
Novikov algebras and their opposite appeared in the 1970s and 1980s
in the papers by Gel'fand and Dorfman \cite{GD} in their study of the Hamiltonian
operator in finite-dimensional mechanics and by Balinskii and Novikov \cite{BN}, see also \cite{N},
in relation with the equations of hydrodynamics. The basis of the free right-commutative algebras as a set of rooted trees
was given by Dzhumadil'daev and L\"ofwall \cite{DL}.
Dzhumadil'daev, Ismailov, and Tulenbaev \cite{DIT}, see also the announcement \cite{DT},
described the free bicommutative algebra $F({\mathfrak B})$ of countable rank and its main numerical invariants.
In particular, they found a basis of $F({\mathfrak B})$ as a $K$-vector space,
computed the Hilbert series of the $d$-generated free bicommutative algebra $F_d({\mathfrak B})$,
the cocharacter and codimension sequences of the variety $\mathfrak B$.

It was shown in \cite{DT} that the square $F^2({\mathfrak B})$ of the algebra $F({\mathfrak B})$ is a commutative and associative algebra.
Therefore one should expect that the algebra $F({\mathfrak B})$ itself has many properties typical for commutative and associative algebras.
The simplest example of a finitely generated bicommutative algebra which is not noetherian is the one-generated free algebra
$F_1({\mathfrak B})$ which has not finitely generated one-sided ideals. But we have established that any finitely generated bicommutative algebra
is weakly noetherian, i.e., satisfies the ascending chain condition for two-sided ideals.

Let $K\{X\}=K\{x_1,x_2,\ldots\}$ be the absolutely free nonassociative algebra and let $f(x_1,\ldots,x_d)\in K\{X\}$.
Recall that the $K$-algebra $R$ satisfies the polynomial identity $f=0$ if
$f(r_1,\ldots,r_d)=0$ for all $r_1,\ldots,r_d\in R$. If $\{f_i\in K\{X\}\mid i\in I\}$ is a set of elements
in $K\{X\}$, then the class $\mathfrak V$ of all algebras satisfying the polynomial identities $f_i=0$, $i\in I$,
is called the {\it variety} defined by the system of polynomial identities $\{f_i\mid i\in I\}$.
The set $T(\mathfrak{V})$ of all polynomial identities satisfied by the variety $\mathfrak{V}$
is called the {\it $T$-ideal} or the {\it verbal ideal} of $\mathfrak{V}$.
By definition,  $T(\mathfrak{V})$ is generated as a T-ideal by any system of polynomials defining the variety $\mathfrak V$.
One of the main problems in the theory of varieties of algebras is:

\begin{problem}{\rm (The finite basis problem, or the Specht problem)}
Can any subvariety of a given variety of algebras
be defined by a finite number of polynomial identities?
\end{problem}

It follows from the description of the cocharacter sequence of the variety $\mathfrak B$ given in \cite{DIT} that in characteristic 0 every T-ideal in
$F({\mathfrak B})$ is generated by its elements in two variables only. Hence the weak noetherian property for finitely generated bicommutative
algebras immediately implies the positive solution to the Specht problem when $\text{char}K=0$. In order to establish a similar result
when the field $K$ is of positive characteristic we apply the classical method of Higman-Cohen \cite{H} and \cite{Co}. Nowadays in many cases
the Specht problem is solved into affirmative using the structure theory of T-ideals
developed by Kemer for associative algebras in characteristic 0 (see his book \cite{K} for an account), its further developments in positive characteristic
(see, e.g., Belov-Kanel, Rowen, and Vishne \cite{BKRV}), and for other classes of algebras
(e.g., Iltyakov \cite{I0} with detailed exposition in \cite{I} for finite dimensional Lie algebras, Vajs and Zel'manov \cite{VZ} for finitely generated Jordan algebras).
But up to the 1970s, the Higman-Cohen method was one of the few methods to handle into affirmative the Specht problem:

\begin{itemize}
\item for groups: Cohen \cite{Co} for ${\mathfrak A}^2$ (below we use the standard notation for the varieties),
Vaughan-Lee \cite{VL1} for ${\mathfrak A}{\mathfrak N}_c\cap{\mathfrak N}_k{\mathfrak A}$;
\item for Lie algebras: Vaughan-Lee \cite{VL2} for $[{\mathfrak A}^2,{\mathfrak E}]$, $\text{char}K\not=2$,
Bryant and Vaughan-Lee \cite{BVL} for ${\mathfrak N}_2{\mathfrak A}$, $\text{char}K\not=2$;
\item for associative algebras in characteristic 0: Latyshev \cite{L1} and Genov \cite{G1, G2} for ${\mathfrak N}_k{\mathfrak A}$,
Latyshev \cite{L2} and Popov \cite{P} for ${\mathfrak N}_k{\mathfrak L}_2$;
\item for associative algebras in positive characteristic: Chiripov and Siderov \cite{ChS} for ${\mathfrak N}_3\mathfrak A$, $\text{char}K\not=2$.
\end{itemize}

\section*{Acknowledgements}
\begin{itemize}
\item This project was carried out when the second named author visited the Institute of Mathematics and Informatics of the Bulgarian Academy of Sciences.
He is very grateful for the creative atmosphere and the warm hospitality during his visit.
\item The authors are thankful to the anonymous referee for the careful reading of the manuscript and the useful suggestions for improving of the exposition.
\item The first named author is deeply obliged to Andreas Weiermann (in the frames of the joint Bulgarian -- Belgium research project
``Mathematical Logic, Algebra, and Algebraic Geometry'' between the Bulgarian Academy of Sciences and the Research Foundation -- Flanders)
for the comments concerning the number of generators of the ideals in Theorems \ref{weak noetherianity},
\ref{Specht in characteristic 0}, and \ref{Specht in any characteristic}, and taken into account in
Remarks \ref{number of generators of ideals}, \ref{generation as T-ideal in char 0}, and \ref{generation as T-ideal in char p}.
\end{itemize}

\section{Preliminaries}

It the sequel we shall denote by $F$ and $F_d$ the free bicommutative algebras $F({\mathfrak B})$ and $F_d({\mathfrak B})$, respectively.
It was established in \cite{DIT} that the following monomials form a basis of the square $F^2_d$ of the algebra $F_d$
as a $K$-vector space:
\begin{equation}\label{basis of F^2}
u_{i,j}=x_{i_1}(\cdots(x_{i_{m-1}}((\cdots((x_{i_m} x_{j_1})x_{j_2})\cdots)x_{j_n}))\cdots),
\end{equation}
where $m,n\geq 1$, $1\leq i_1\leq\cdots\leq i_{m-1}\leq i_m\leq d$, $1\leq j_1\leq j_2\leq\cdots\leq j_n\leq d$.
If $L_{x_i}$ and $R_{x_j}$ are the operators of left and right multiplication on $F_d$, defined respectively by
\[
L_{x_i}:u\to x_iu\text{ and }R_{x_j}:u\to ux_j,\quad u\in F_d,
\]
then (\ref{basis of F^2}) can be written as
\[
u_{i,j}=L_{x_{i_1}}\cdots L_{x_{i_{m-1}}}R_{x_{j_n}}\cdots R_{x_{j_1}}(x_{i_m}).
\]
For any permutations $\sigma\in S_m$ and $\tau\in S_n$ the element $u_{i,j}$ from (\ref{basis of F^2}) satisfies the equality
\[
u_{i,j}=x_{i_{\sigma(1)}}(\cdots(x_{i_{\sigma(m-1)}}((\cdots((x_{i_{\sigma(m)}} x_{j_{\tau(1)}})x_{j_{\tau(2)}})\cdots)x_{j_{\tau(n)}}))\cdots),
\]
i.e.,
\[
u_{i,j}=L_{x_{i_{\sigma(1)}}}\cdots L_{x_{i_{\sigma(m-1)}}}R_{x_{j_{\tau(n)}}}\cdots R_{x_{j_{\tau(1)}}}(x_{i_{\sigma(m)}}).
\]
By \cite{DIT} the algebra $F_d$ is isomorphic to the following algebra $F(d)$.
Let $\mathbb{N}_0$ be the set of all non-negative integers and let $\mathbb{N}_0^d=\mathbb{N}_0\times\cdots\times \mathbb{N}_0$
be the direct sum of $d$ copies of $\mathbb{N}_0$. Let $\varepsilon_i=(0,\ldots,0,1,0,\ldots,0)\in \mathbb{N}_0^d$,
where all components except the $i$-th are equal to 0. The algebra $F(d)$ has a basis
\begin{equation}\label{basis of F(d)}
\{x_i\mid i=1,\ldots,d\}\cup \{v_{\alpha, \beta} \mid\alpha, \beta\in \mathbb{N}_0^d, \sum_{i=1}^d\alpha_i>0, \sum_{i=1}^d\beta_j> 0 \}
\end{equation}
and multiplication $\circ$ given by the following rules:
\begin{equation}\label{multiplication in F(d)}
\begin{array}{c}
x_i\circ x_j=v_{\varepsilon_i, \varepsilon_j},\\
x_i\circ v_{\alpha, \beta}=v_{\alpha+\varepsilon_i, \beta},\\
v_{\alpha, \beta}\circ x_j=v_{\alpha, \beta+\varepsilon_j},\\
v_{\alpha, \beta}\circ v_{\gamma, \delta}=v_{\alpha+\gamma, \beta+\delta}.\\
\end{array}
\end{equation}
The isomorphism $\pi:F_d\to F(d)$ is defined on the basis monomials of $F_d$ in the following way
and then extended by linearity.
We associate to any monomial $u_{i,j}$ in (\ref{basis of F^2}) the $m$-tuple
$X_d^{(i)}=(x_{i_1},\ldots ,x_{i_{m-1}},x_{i_m})$ and the $n$-tuple $X_d^{(j)}=(x_{j_1},x_{j_2},\ldots,x_{j_n})$. If
\[
X_d^{(i)}=(\underbrace{x_1,\ldots,x_1}_{\alpha_1\text{ times}},\ldots,\underbrace{x_d,\ldots,x_d}_{\alpha_d\text{ times}}),\quad
X_d^{(j)}=(\underbrace{x_1,\ldots,x_1}_{\beta_1\text{ times}},\ldots,\underbrace{x_d,\ldots,x_d}_{\beta_d\text{ times}}),
\]
then $\pi(u_{i,j})=v_{\alpha,\beta}$.

For our purposes it is more convenient to identify the element $v_{\alpha,\beta}$ with the monomial
\[
Y_d^{\alpha}Z_d^{\beta}=y_1^{\alpha_1}\cdots y_d^{\alpha_d}z_1^{\beta_1}\cdots z_d^{\beta_d}
\]
in the polynomial algebra $K[Y_d,Z_d]=K[y_1,\ldots,y_d,z_1,\ldots,z_d]$ in $2d$ commutative and associative variables. In this notation
the algebra $F(d)$ is isomorphic to the algebra $G_d$ with basis
\begin{equation}\label{basis of G_d}
\{x_i\mid i=1,\ldots,d\}\cup \{Y_d^{\alpha}Z_d^{\beta} \mid\deg Y_d^{\alpha}, \deg Z_d^{\beta}> 0 \}
\end{equation}
and multiplication
\begin{equation}\label{multiplication in G_d}
\begin{array}{c}
x_i\circ x_j=y_iz_j,\\
\\
x_i\circ Y_d^{\alpha}Z_d^{\beta}=y_iY_d^{\alpha}Z_d^{\beta},\\
\\
Y_d^{\alpha}Z_d^{\beta}\circ x_j=Y_d^{\alpha}Z_d^{\beta}z_j,\\
\\
Y_d^{\alpha}Z_d^{\beta}\circ Y_d^{\gamma}Z_d^{\delta}=Y_d^{\alpha+\gamma}Z_d^{\beta+\delta}.\\
\end{array}
\end{equation}

The following lemma summarizes the properties of $G_d$ stated above.

\begin{lemma}\label{properties of G_d} In the notation of (\ref{basis of G_d}) and (\ref{multiplication in G_d}):

{\rm (i)} The algebra $F_d$ is isomorphic to the algebra $G_d$ generated by $X_d=\{x_1,\ldots,x_d\}$. The square $G_d^2$ of $G_d$ has a basis
$\{Y_d^{\alpha}Z_d^{\beta} \mid\deg Y_d^{\alpha}, \deg Z_d^{\beta}> 0 \}$.

{\rm (ii)} The left and the right multiplications by the elements of $X_d$ on $G^2_d$ define on it a natural structure of a $K[Y_d,Z_d]$-module.
\end{lemma}

As an immediate consequence of Lemma \ref{properties of G_d} we obtain the following description of the algebra $F$.

\begin{lemma}\label{properties of G} The algebra $F$ is isomorphic to the algebra $G$ with basis
\[
X\cup\{Y^{\alpha}Z^{\beta}\in K[Y,Z]=K[y_1,y_2,\ldots,z_1,z_2,\ldots]\mid\deg Y^{\alpha},\deg Z^{\beta}>0\}.
\]
The algebra $G$ is generated by $X$ and the left and the right multiplications by the elements from $X$ on $G^2$ make it
a $K[Y,Z]$-module.
\end{lemma}

\section{Weak noetherianity}
We start with an example showing that finitely generated bicommutative algebras are not necessarily noetherian.
\begin{proposition}\label{not noetherian}
The free bicommutative algebra $F_1=F_1({\mathfrak B})$ is not noetherian.
\end{proposition}

\begin{proof}
By Lemma \ref{properties of G_d} the algebras $F_1$ and $G_1$ are isomorphic and we shall work in $G_1$ instead of in $F_1$.
As a vector space $G_1$ has a basis
\[
\{x_1\}\cup\{y_1^{\alpha_1}z_1^{\beta_1}\mid\alpha_1,\beta_1>0\}.
\]
Consider the left ideal $I$ of $G_1$ generated by the monomials
\begin{equation}\label{generating set}
y_1z_1^{\delta},\quad \delta=1,2,\ldots.
\end{equation}
If $I$ is finitely generated, then it can be generated by a finite number of monomials $y_1z_1^{\delta}$, $1\leq\delta\leq n$, from (\ref{generating set}).
Then, by (\ref{multiplication in G_d}), $I$ is spanned by the monomials
\[
\underbrace{x_1\circ\cdots\circ x_1}_{k\text{ times}}\circ y_1z_1^{\delta}=y_1^{k+1}z_1^{\delta},\quad k=0,1,2,\ldots,\quad1\leq\delta\leq n,
\]
\[
y_1^{\alpha_1}z_1^{\beta_1}\circ y_1z_1^{\delta}=y_1^{\alpha_1+1}z_1^{\beta_1+\delta},\quad\alpha_1,\beta_1\geq 1,\quad 1\leq\delta\leq n.
\]
Obviously, this list of monomials does not contain the monomials $y_1z_1^{\delta}$ from $I$ for $\delta>n$, i.e.,
the left ideal $I$ is not finitely generated. The considerations for not finitely generated right ideals of $G_1$ are similar.
It is sufficient to consider the right ideal generated by $y_1^{\gamma}z_1$, $\gamma=1,2,\ldots$.
\end{proof}

The following theorem is the first main result of our paper.

\begin{theorem}\label{weak noetherianity}
Finitely generated bicommutative algebras satisfy the ascending chain condition for two-sided ideals.
\end{theorem}

\begin{proof}
It is sufficient to work in the free algebra $F_d$, or, equivalently, in its isomorphic copy $G_d$.
The factor algebra $G_d/G_d^2$ is finite dimensional and hence noetherian. Therefore the theorem will be established if
we prove the weak noetherianity for the ideals in $G_d^2$. Every two-sided ideal $I$ of $G_d$ which is in $G_d^2$ is stable under the
left and right multiplications by the generators $x_1,\ldots,x_d$ of $G_d$ and hence is a $K[Y_d,Z_d]$-submodule of $F_d^2$.
As a $K[Y_d,Z_d]$-module $F_d^2$ is generated by
the finite number of monomials $y_iz_j$, $i,j=1,\ldots,d$. Hence the $K[Y_d,Z_d]$-submodule $I$ of $G_d^2$ is also finitely generated
which implies that $I$ is finitely generated also as a two-sided ideal.
\end{proof}

\begin{remark}\label{number of generators of ideals}
By Theorem \ref{weak noetherianity} every two-sided ideal $I$ of the free becommutative algebra $F_d$ is finitely generated.
It is an interesting problem how the number of the generators of $I$ depends on the rank $d$ of $F_d$. Since the square $F^2_d$ of $F_d$
is commutative and associative, we shall comment what happens for the ideals of the polynomial algebra $K[X_d]$. Seidenberg \cite{S}
formalized the problem in the following way. Given a function $f: {\mathbb N}_0\to{\mathbb N}_0$, what is the maximal $k_0\in\mathbb N$ with the property:
There exists an ideal $I$ of $K[X_d]$ generated by
the set $\{p_0(X_d),p_1(X_d),\ldots,p_{k_0}(X_d)\}$ such that $\deg(p_k(X_d))\leq f(k)$, $k=0,1,\ldots,k_0$, and the chain of ideals
$I_0\subset I_1\subset\cdots\subset I_{k_0}$, where $I_k$ is generated by $\{p_0(X_d),p_1(X_d),\ldots,p_k(X_d)\}$, is strictly increasing.
He showed that there exists a bound $g_f$ depending on $f$ and $d$ only which is recursive in $f$ for a fixed $d$.
Moreno-Soc\'{\i}as \cite{MS2} found a simpler bound $g'_f$ which is primitive recursive in $f$ for all $d$ but there is no bound
which is primitive recursive in $d$ in general. In particular \cite{MS1}, he constructed an example of an ideal $I$ for the function $f(k)=m+k$, $k=0,1,\ldots$,
and $d\geq 1$, such that the exact bound of the number $k_0$ of generators is
$k_0=A(d,m-1)-1$, where $A$ is the Ackermann function \cite{A} which is recursive and known that grows faster than any primitive recursive function.
\end{remark}

\section{The Specht property}
One of the most important numerical invariants of a given variety $\mathfrak V$ of algebras over a filed of characteristic 0 is its cocharacter sequence
\[
\chi_n({\mathfrak V})=\sum_{\lambda\vdash n}m(\lambda)\chi_{\lambda},\quad n=1,2,\ldots.
\]
Here $\chi_n({\mathfrak V})$ is the $S_n$-character of the vector space $P_n({\mathfrak V})$ of the multilinear elements
of degree $n$ in the free algebra $F_n({\mathfrak V})$ of the variety $\mathfrak V$ under the natural left action of the symmetric group $S_n$.
We have denoted by $\chi_{\lambda}$ the irreducible $S_n$-character indexed
with the partition $\lambda$ of $n$ and $m(\lambda)\in{\mathbb N}_0$ is the multiplicity
of $\chi_{\lambda}$ in $\chi_n({\mathfrak V})$. For the variety $\mathfrak B$ of bicommutative algebras it was shown in \cite{DIT} that
\[
\chi_n({\mathfrak B})=\sum_{(\lambda_1,\lambda_2)}m(\lambda_1,\lambda_2)\chi_{(\lambda_1,\lambda_2)},
\]
where $\lambda=(\lambda_1,\lambda_2)$ is a partition in two parts,
with explicitly given values of $m(\lambda_1,\lambda_2)$.
This description, together with Theorem \ref{weak noetherianity} easily implies the positive solution to the Specht problem in characteristic 0.

\begin{theorem}\label{Specht in characteristic 0}
Let $\mathfrak V$ be any variety of bicommutative algebras over a field $K$ of characteristic $0$.
Then $\mathfrak V$ can be defined by a finite system of polynomial identities.
\end{theorem}

\begin{proof}
Let the base field $K$ be of characteristic 0 and let $\mathfrak V$ be a variety of $K$-algebras.
It is well known (see, e.g., \cite[Chapter 12]{D1})
that if the nonzero multiplicities $m(\lambda)$ in the cocharacter sequence $\chi_n({\mathfrak V})$ are only for partitions
$\lambda=(\lambda_1,\ldots,\lambda_d)$ in not more than $d$ parts, then every subvariety $\mathfrak W$ of $\mathfrak V$ can be defined by polynomial identities
$f(x_1,\ldots,x_d)=0$, where $f(x_1,\ldots,x_d)\in F_d({\mathfrak V})$. In our case, if $\mathfrak V$ is a subvariety of $\mathfrak B$, it can be defined by its
identities from $F_2({\mathfrak B})$. Hence the T-ideal $T({\mathfrak V})\subseteq F({\mathfrak B})$ of the polynomial identities
of $\mathfrak V$ is generated as a T-ideal by its elements in $T({\mathfrak V})\cap F_2({\mathfrak B})$.
The variety of all bicommutative algebras is defined by two identities. Hence to show that the variety $\mathfrak V$
has a finite basis of polynomial identities in the absolutely free algebra $K\{X\}$ it is sufficient to show that $T({\mathfrak V})\cap F_2({\mathfrak B})$
is finitely generated as a T-ideal in $F_2({\mathfrak B})$.
Now Theorem \ref{weak noetherianity} gives the much stronger result that $T({\mathfrak V})\cap F_2({\mathfrak B})$
is finitely generated as an ordinary two-sided ideal.
\end{proof}

\begin{remark}\label{generation as T-ideal in char 0}
As in Remark \ref{number of generators of ideals} we can ask how many polynomial identities we need to define a subvariety $\mathfrak V$ of
$\mathfrak B$ in the case of characteristic 0. In the recent paper \cite{D2} one of the authors has shown that
if $\mathfrak V$ satisfies a polynomial identity of degree $k$, then the number of the irreducible $S_n$-components in
the $S_n$-module $P_n({\mathfrak V})$ of the multilinear elements in $F({\mathfrak V})$ is bounded by $3k^2$. One can derive from here
that a similar bound holds for the number of the generators of the T-ideal $T({\mathfrak V})$ in $F({\mathfrak B})$ if
$\mathfrak V$ satisfies an identity of degree $k$ and this bound does not depend on the degree of the other identities satistied by $\mathfrak V$.
\end{remark}

For the solution to the Specht problem over a field $K$ of arbitrary characteristic we shall apply the Higman-Cohen method
based on the technique of partially ordered sets.

\begin{definition}
The partially ordered set $(T, \preceq)$ is called {\it partially well-ordered} if for every subset $S$ of $T$ there is a finite subset $S_0$ of $S$
with the property that for each $s\in S$ there is an element $s_0\in S_0$ such that $s_0\preceq s$.
\end{definition}

Let $\Phi$ be the set of all order preserving maps $\varphi:{\mathbb N}\to{\mathbb N}$, i.e.,
if $i<j$, $i,j\in\mathbb N$, then $\varphi(i)<\varphi(j)$. Let $(T,\preceq)$ be a partially ordered set
and let $\overline{T}$ be the set of all finite sequences of elements in $T$.
Define the following partial order on $\overline{T}$. If $(a_1,\ldots,a_m)$ and $(b_1,\ldots,b_n)$ are two sequences, then
\[
(a_1,\ldots,a_m)\preceq(b_1,\ldots,b_n)
\]
if and only if there exists $\varphi\in\Phi$ such that
\[
a_i\preceq b_{\varphi(i)}\text{ for all }i=1,\ldots,m.
\]
One of the key ingredients of the Higman-Cohen method is the following result.

\begin{proposition}\label{partially well ordered}{\rm \cite[Theorem 4.3]{H}}
If $(T, \preceq)$ is a partially well-ordered set, then the set $(\overline{T}, \preceq)$ is also partially well-ordered.
\end{proposition}

For our purposes we shall need a restatement of a partial case of \cite[Lemma 1]{BVL}.
We shall include the proof for completeness of the exposition.

\begin{proposition}\label{partially well ordered monomials}{\rm \cite[Lemma 1]{BVL}}
Let $([Y,Z],\preceq)$ be the set of all commutative and associative monomials
in the variables $Y=\{y_1,y_2,\ldots\}$ and $Z=\{z_1,z_2,\ldots\}$ with the following partial order:
\[
Y^{\alpha}Z^{\beta}=y_1^{\alpha_1}\cdots y_m^{\alpha_m}z_1^{\beta_1}\cdots z_n^{\beta_n}\preceq
y_1^{\gamma_1}\cdots y_p^{\gamma_p}z_1^{\delta_1}\cdots z_q^{\delta_q}=Y^{\gamma}Z^{\delta}
\]
if and only if there exists $\varphi\in\Phi$ such that the monomial
\[
\varphi(Y^{\alpha}Z^{\beta})=y_{\varphi(1)}^{\alpha_1}\cdots y_{\varphi(m)}^{\alpha_m}z_{\varphi(1)}^{\beta_1}\cdots z_{\varphi(n)}^{\beta_n}
\]
divides the monomial $Y^{\gamma}Z^{\delta}$. Then the set $([Y,Z],\preceq)$ is partially well-ordered.
\end{proposition}

\begin{proof}
If $m<n$ in the monomial $Y^{\alpha}Z^{\beta}=y_1^{\alpha_1}\cdots y_m^{\alpha_m}z_1^{\beta_1}\cdots z_n^{\beta_n}$
we may multiply it by $y_{m+1}^0\cdots y_n^0$ and similarly if $n<m$.
Hence we may assume that $m=n$ in $Y^{\alpha}Z^{\beta}$ and identify it with the sequence
\[
\overline{Y^{\alpha}Z^{\beta}}=(y^{\alpha_1}z^{\beta_1},\ldots,y^{\alpha_m}z^{\beta_m}).
\]
We partially order the set of monomials $[y,z]=\{y^az^b\mid a,b\geq 0\}$ by divisibility:
\[
y^az^b\preceq y^cz^d\text{ if and only if }y^az^b\text{ divides }y^cz^d.
\]
The sets $(\overline{[y,z]},\preceq)$ and $([Y,Z],\preceq)$ are isomorphic as partially ordered sets:
If $\varphi\in\Phi$ is an order preserving map then $y^{\alpha_i}z^{\beta_i}$ divides $y^{\gamma_{\varphi(i)}}z^{\delta_{\varphi(i)}}$
for all $i=1,\ldots,m$ if and only if the monomial
$\varphi(Y^{\alpha}Z^{\beta})=y_{\varphi(1)}^{\alpha_1}\cdots y_{\varphi(m)}^{\alpha_m}z_{\varphi(1)}^{\beta_1}\cdots z_{\varphi(m)}^{\beta_m}$
divides the monomial $Y^{\gamma}Z^{\delta}$.
Clearly, the set $([y,z],\preceq)$ is partially well-ordered. Applying Proposition \ref{partially well ordered monomials} we derive that the set
$(\overline{[y,z]},\preceq)$ is also partially well-ordered. Hence the same holds for the set $([Y,Z],\preceq)$.
\end{proof}

In the sequel we shall consider the set $[Y,Z]$ equipped with the above partial order $\preceq$.
Now we shall equip $[Y,Z]$ with one more linear order $\leq$ which is a version of the reflected lexicographic order:
\[
Y^{\alpha}Z^{\beta}=y_1^{\alpha_1}\cdots y_m^{\alpha_m}z_1^{\beta_1}\cdots z_m^{\beta_m}<
y_1^{\gamma_1}\cdots y_m^{\gamma_m}z_1^{\delta_1}\cdots z_m^{\delta_m}=Y^{\gamma}Z^{\delta}
\]
if and only if
\begin{itemize}
\item $\alpha_i<\gamma_i$ for some $i$ and $\alpha_j=\gamma_j$ for $j=i+1,\ldots,m$,
\item or $Y^{\alpha}=Y^{\gamma}$, $\beta_i<\delta_i$ for some $i$ and $\beta_j=\delta_j$ for $j=i+1,\ldots,m$.
\end{itemize}
Obviously the set $([Y,Z],\leq)$ is well-ordered. If
\[
f(Y,Z)=\sum_{i=1}^k\vartheta_{\alpha^{(i)},\beta^{(i)}}Y^{\alpha^{(i)}}Z^{\beta^{(i)}},\quad
0\not=\vartheta_{\alpha^{(i)},\beta^{(i)}}\in K,
\]
and $Y^{\alpha^{(1)}}Z^{\beta^{(1)}}>\cdots >Y^{\alpha^{(k)}}Z^{\beta^{(k)}}$, then we call the monomial
\[
\text{wt}(f(Y,Z))=Y^{\alpha^{(1)}}Z^{\beta^{(1)}}
\]
the {\it weight} of $f(Y,Z)$. It is easy to see, that if $f(Y,Z)$ belongs to the square $G^2$ of the algebra $G$, then
\begin{equation}\label{multiplication of weight}
\text{wt}(x_i\circ f(Y,Z))=y_i\text{wt}(f(Y,Z)),\quad \text{wt}(f(Y,Z)\circ x_i)=\text{wt}(f(Y,Z))z_i,
\end{equation}
\begin{equation}\label{action on weight by varphi}
\text{wt}(\varphi(f(Y,Z)))=\text{wt}(f(y_{\varphi(1)},\ldots,y_{\varphi(m)},z_{\varphi(1)},\ldots,z_{\varphi(n)}))
=\varphi(\text{wt}(f(Y,Z)))
\end{equation}
for all $\varphi\in\Phi$.

The following lemma is the key step in the proof of the Specht property for varieties of bicommutative algebras over an arbitrary field $K$.

\begin{lemma}\label{key lemma}
Let $f$ and $g$ be two polynomials in the square $F^2$ of the free bicommutative algebra $F({\mathfrak B})$
and let $f_1(Y,Z)$, $g_1(Y,Z)$ be their images in the algebra $G$. If
$\text{\rm wt}(f_1)\preceq \text{\rm wt}(g_1)$, then there is an element $h$ in the $T$-ideal of $F$ generated by $f$
with image $h_1$ in $G$ such that $\text{\rm wt}(h_1)=\text{\rm wt}(g_1)$.
\end{lemma}

\begin{proof}
Let
\[
\text{\rm wt}(f_1)=Y^{\alpha}Z^{\beta}, \quad\text{\rm wt}(g_1)=Y^{\gamma}Z^{\delta},
\]
and let $\varphi\in\Phi$ be such that
\[
\varphi(\text{\rm wt}(f_1))=\varphi(Y^{\alpha}Z^{\beta})
=y_{\varphi(1)}^{\alpha_1}\cdots y_{\varphi(m)}^{\alpha_m}z_{\varphi(1)}^{\beta_1}\cdots z_{\varphi(n)}^{\beta_n}
\]
divides the monomial $Y^{\gamma}Z^{\delta}$. Hence there exists a monomial $Y^{\xi}Z^{\eta}\in[Y,Z]$ such that
\[
Y^{\xi}Z^{\eta}\varphi(\text{\rm wt}(f_1))=\text{\rm wt}(g_1).
\]
Since $\varphi$ defines an endomorphism of the algebra $F$ and T-ideals are closed under endomorphisms, we obtain that the polynomial
$\varphi(f)$ belongs to the T-ideal $(f)^T\subset F$ generated by $f$. Its image in $G^2$ is $\varphi(f_1)$. By (\ref{action on weight by varphi})
$\text{wt}(\varphi(f_1))=\varphi(\text{wt}(f_1))$. The weight of $\text{wt}(g_1)$ can be obtained from $\varphi(\text{wt}(f_1))$
by consequent right- and left-multiplications by elements of $X$. The same multiplications produce a polynomial
$h\in (f)^T$. By (\ref{multiplication of weight}) $Y^{\xi}Z^{\eta}\varphi(\text{\rm wt}(f_1))=\text{\rm wt}(h_1)=\text{\rm wt}(g_1)$.
\end{proof}

The following theorem is the second main result of the paper.

\begin{theorem}\label{Specht in any characteristic}
Any variety $\mathfrak V$ of bicommutative algebras over a field $K$ of arbitrary characteristic has a finite basis of its polynomial identities.
\end{theorem}

\begin{proof}
Again, since the variety of all bicommutative algebras is defined by two identities,
it is sufficient to establish the finitely generation of the T-ideals in $F({\mathfrak B})$. The algebra $F({\mathfrak B})/F^2({\mathfrak B})$ is the free algebra
of the variety ${\mathfrak N}_2$ of all algebras with trivial multiplication defined by the polynomial identity $x_1x_2=0$.
Clearly, ${\mathfrak N}_2$ has no proper subvarieties. Hence, for the proof of the theorem it is sufficient to consider varieties $\mathfrak V$
with T-ideals $I=T({\mathfrak V})$ in $F^2({\mathfrak B})$. Let $I_1$ be the image of $I$ in the square $G^2$ of the algebra $G$. The set
$([Y,Z],\preceq)$ is partially well-ordered. Hence there is a finite set of polynomials $\{f^{(1)},\ldots,f^{(k)}\}$ in $I$ with images
$\{f^{(1)}_1,\ldots,f^{(k)}_1\}$ in $I_1$ with the property that for any $g\in I$ with image $g_1\in I_1$ there exists an $f^{(i)}$
such that $\text{wt}(f^{(i)}_1)\preceq \text{wt}(g_1)$. Let us assume that the polynomials $f^{(1)},\ldots,f^{(k)}$ do not generate the T-ideal $I$.
Since the set $([Y,Z],\leq)$ is well-ordered, there is a polynomial $g\in I\setminus (f^{(1)},\ldots,f^{(k)})^T$
such that the weight $\text{wt}(g_1)$ of its image in $I_1$ is minimal in the set $\text{wt}(I\setminus (f^{(1)},\ldots,f^{(k)})^T)_1$.
If $f^{(i)}$ is such that $\text{wt}(f^{(i)}_1)\preceq \text{wt}(g_1)$, then by Lemma \ref{key lemma} there exists an $h\in (f^{(i)})^T\subseteq I$
such that its image $h_1\in I_1$ satisfies $\text{wt}(h_1)=\text{wt}(g_1)$. If the coefficients of $\text{wt}(h_1)$ and $\text{wt}(g_1)$
in $h_1$ and $g_1$ are, respectively, $\mu$ and $\nu$, then the polynomial $\displaystyle u=g-\frac{\mu}{\nu}h$ belongs to $I$. If $u\not=0$,
then $u\in I\setminus (f^{(1)},\ldots,f^{(k)})^T$. But this is impossible because $\text{wt}(u_1)\prec\text{wt}(g_1)$ for its image $u_1\in G^2$
which contradicts to the minimality of $\text{wt}(g_1)$.
\end{proof}

\begin{remark}\label{generation as T-ideal in char p}
Again, as in Remarks \ref{number of generators of ideals} and \ref{generation as T-ideal in char 0}
we can ask about the number of generators of the T-ideals of $F({\mathfrak B})$ when the base field is of positive characteristic.
Since we work in the polynomial algebra $K[Y,Z]$ considered as a $K[X]$-bimodule we shall mention several results concerning the number of generators,
theory of Gr\"obner bases and other algorithmic problems:
Aschenbrenner, Hillar \cite{AH}, Hillar, Windfeldt \cite{HW}, Hillar, Sullivant \cite{HS}, Krone \cite{Kr}, and
Hillar, Krone, Leykin \cite{HKL}.
\end{remark}


\begin{thebibliography}{99}

\bibitem{A}
W. Ackermann,
{\it Zum Hilbertschen Aufbau der reellen Zahlen},
Math. Ann. {\bf 99} (1928), 118-133.

\bibitem{AH}
M. Aschenbrenner, C.J. Hillar,
{\it Finite generation of symmetric ideals},
Trans. Amer. Math. Soc.  {\bf 359}  (2007),  No. 11, 5171-5192.
Erratum: Trans. Amer. Math. Soc. {\bf 361}  (2009), No. 10, 5627.

\bibitem{BN}
A.A. Balinskii, S.P. Novikov,
{\it Poisson brackets of hydrodynamic type, Frobenius algebras and Lie algebras} (Russian),
Dokl. Akad. Nauk SSSR {\bf 283} (1985), No. 5, 1036-1039.
Translation: Sov. Math., Dokl. {\bf 32} (1985), 228-231.

\bibitem{BKRV}
A. Belov-Kanel, L. Rowen, U. Vishne,
{\it Full exposition of Specht's problem},
Serdica Math. J. {\bf 38} (2012), Nos 1-3, 313-370.

\bibitem{BVL}
R.M. Bryant, M.R. Vaughan-Lee,
{\it Soluble varieties of Lie algebras},
Quart. J. Math., Oxford Ser. (2) {\bf 23} (1972), 107-112.

\bibitem{Ca}
A. Cayley,
{\it On the theory of analytical forms called trees},
Phil. Mag. {\bf 13} (1857), 172-176. Collected Math. Papers, University Press, Cambridge, Vol. 3, 1890, 242-246.

\bibitem{ChS}
P.Zh. Chiripov, P.N. Siderov,
{\it On the bases of identities of some varieties of associative algebras} (Russian),
PLISKA, Stud. Math. Bulg. {\bf 2} (1981), 103-115.

\bibitem{Co}
D.E. Cohen,
{\it On the laws of a metabelian variety},
J. Algebra {\bf 5} (1967), 267-273.

\bibitem{D1}
V. Drensky,
{\it Free Algebras and PI-Algebras. Graduate Course in Algebra},
Springer-Verlag, Singapore, 2000.

\bibitem{D2}
V. Drensky,
{\it Varieties of bicommutative algebras},
arXiv:1706.04279v1 [math.RA].

\bibitem{DIT}
A.S. Dzhumadil'daev, N.A. Ismailov, K.M. Tulenbaev,
{\it Free bicommutative algebras},
Serdica Math. J. {\bf 37} (2011), No. 1, 25-44.

\bibitem{DL}
A. Dzhumadil'daev, C. L\"ofwall,
{\it Trees, free right-symmetric algebras, free Novikov algebras and identities},
Homology Homotopy Appl. {\bf 4} (2002), 165-190.

\bibitem{DT}
A.S. Dzhumadil'daev, K.M. Tulenbaev,
{\it Bicommutative algebras} (Russian),
Usp. Mat. Nauk {\bf 58} (2003), No. 6, 149-150.
Translation: Russ. Math. Surv. {\bf 58} (2003), No. 6, 1196-1197.

\bibitem{GD}
I.M. Gel'fand, I.Ya. Dorfman,
{\it Hamiltonian operators and algebraic structures structures related to them} (Russian),
Funktsional. Anal. i Prilozhen. {\bf 13} (1979), No. 4, 13-30.
Translation: Funct. Anal. Appl. {\bf 13} (1980), 248-262.

\bibitem{G1}
G.K. Genov,
{\it On the Spechtianness of certain varieties of associative algebras over a field of characteristic zero} (Russian),
C. R. Acad. Bulg. Sci. {\bf 29} (1976), 939-941.

\bibitem{G2}
G.K. Genov,
{\it Some Specht varieties of associative algebras} (Russian),
PLISKA, Stud. Math. Bulg. {\bf 2} (1981), 30-40.

\bibitem{H}
G. Higman,
{\it Ordering by divisibility in abstract algebras},
Proc. Lond. Math. Soc., III. Ser. {\bf 2} (1952), 326-336.

\bibitem{HKL}
C.J. Hillar, R. Krone, A. Leykin,
{\it Equivariant Gr\"obner bases},
arXiv:1610.02075v2 [math.AC].

\bibitem{HS}
C.J. Hillar, S. Sullivant,
{\it Finite Gr\"obner bases in infinite dimensional polynomial rings and applications},
Adv. Math. {\bf 229} (2012), No. 1, 1-25.

\bibitem{HW}
C.J. Hillar, T.  Windfeldt,
{\it Minimal generators for symmetric ideals},
Proc. Amer. Math. Soc. {\bf 136}  (2008),  No. 12, 4135-4137.

\bibitem{I0}
A.V. Il'tyakov,
{\it On finite basis of identities of Lie algebra representations},
Nova J. Algebra Geom. {\bf 1} (1992), No. 3, 207-259.

\bibitem{I}
A.V. Iltyakov,
{\it Polynomial identities of finite dimensional Lie algebras},
Research Report 98-6, Univ. of Sydney,
http://www.maths.usyd.edu.au/res/Algebra/Ilt/1998-6.html.

\bibitem{K}
A.R. Kemer,
{\it Ideals of Identities of Associative Algebras},
Translations of Math. Monographs {\bf 87}, AMS, Providence, RI, 1991.

\bibitem{Kr}
R. Krone,
{\it Equivariant Gr\"obner bases of symmetric toric ideals},
Proceedings of the 2016 ACM International Symposium on Symbolic and Algebraic Computation,
311-318, ACM, New York, 2016.

\bibitem{L1}
V.N. Latyshev,
{\it Partially ordered sets and nonmatrix identities of associative algebras} (Russian),
Algebra Logika {\bf 15} (1976), 53-70.
Translation: Algebra Logic {\bf 15} (1976), 34-45.

\bibitem{L2}
V.N. Latyshev,
{\it Finite basis property of identities of certain rings} (Russian),
Usp. Mat. Nauk {\bf 32} (1977), No. 4(196), 259-260.

\bibitem{MS1}
G. Moreno Soc\'{\i}as,
{\it An Ackermannian polynomial ideal},
Applied algebra, algebraic algorithms and error-correcting codes
(New Orleans, LA, 1991), 269-280,
Lecture Notes in Comput. Sci. {\bf 539}, Springer-Verlag, Berlin, 1991.

\bibitem{MS2}
G. Moreno Soc\'{\i}as,
{\it Length of polynomial ascending chains and primitive recursiveness},
Math. Scand.  {\bf 71}  (1992), No. 2, 181-205.

\bibitem{N}
S.P. Novikov,
{\it Geometry of conservative systems of hydrodynamic type. The method of averaging for field-theoretical systems} (Russian),
International Conference on Current Problems in Algebra and Analysis (Moscow-Leningrad, 1984),
Uspekhi Mat. Nauk {\bf 40} (1985), No. 4(244), 79-89.
Translation: Russ. Math. Surv. {\bf 40} (1985), No. 4, 85-98.

\bibitem{P}
A.P. Popov,
{\it On the Specht property of some varieties of associative algebras} (Russian),
PLISKA, Stud. Math. Bulg. {\bf 2} (1981), 41-53.

\bibitem{S}
A. Seidenberg,
{\it On the length of a Hilbert ascending chain},
Proc. Amer. Math. Soc. {\bf 29} (1971), 443-450.

\bibitem{VZ}
A.Ya. Vajs, E.I. Zel'manov,
{\it Kemer's theorem for finitely generated Jordan algebras} (Russian),
Izv. Vyssh. Uchebn. Zaved., Mat. (1989), No. 6(325), 42-51.
Translation: Sov. Math. {\bf 33} (1990), No. 6, 38-47.

\bibitem{VL1}
M.R. Vaughan-Lee,
{\it Abelian by nilpotent varieties},
Quart. J. Math. Oxford Ser. (2) {\bf 21} (1970), 193-202.

\bibitem{VL2}
M.R. Vaughan-Lee,
{\it Centre-by-metabelian Lie algebras},
J. Aust. Math. Soc. {\bf 15} (1973), 259-264.

\end{thebibliography}
\end{document}